\documentclass[12pt,a4paper]{article}
\usepackage{amsfonts, amsthm, amsmath, amssymb}
\usepackage[english]{babel}

\theoremstyle{plain}
\newtheorem{theorem}{Theorem}
\newtheorem*{theorem*}{Theorem}

\newtheorem{lemma}[theorem]{Lemma}

\newtheorem{prop}[theorem]{Proposition}

\theoremstyle{definition}

\theoremstyle{remark}

\newtheorem{example}[theorem]{Example}

\numberwithin{equation}{section}

\date{}

\title{Weak approximation of an invariant measure and a low boundary
of the entropy, II}

\author{B.M. Gurevich\thanks{Email: gurevicu@mech.math.msu.su The work is
supported in part by the RFBR grant 14-01-00379 a.}\\
\small{Moscow State University and Institute for Information
Transmission}
Problems\\
\small {of Russian Academy of Sciences}}

\begin{document}

\maketitle

\thispagestyle{empty}

\begin{abstract}
For a measurable map $T$ and a sequence of $T$-invariant probability
measures $\mu_n$ that converges in some sense to a $T$-invariant
probability measure $\mu$, an estimate from below for the
Kolmogorov--Sinai entropy of $T$ with respect to $\mu$ is suggested
in terms of the entropies of $T$ with respect to $\mu_1$, $\mu_2$,
\dots. This result is obtained under the assumption that some
generating partition has finite entropy. By an explicite example it
is shown that, in general, this assumption cannot be removed.
\end{abstract}
\medskip

\qquad\ \small{Keywords: invariant measure, metric entropy, Markov
shift}

\bigskip

\section{\bf Introduction}
\label{Introd}
\medskip

In problems of ergodic theory and thermodynamic formalism it is
sometimes necessary to estimate the entropy of a measure preserving
map. If this map acts in a compact metric space and is expansive,
one can use the fact that the entropy is semi-continuous from above
on the space of invariant probability measures with weak topology.
Both conditions --- the compactness and expansiveness are essential
in this context, and if at least one of them fails, one has to use
other tools. Two such tools are suggested in this paper (see
theorems \ref{main} and \ref{two}). In both cases a countable
generating partition appears. In Theorem \ref{main} we assume that
the entropy of this partition with respect to the measure of
interest is finite and in Section \ref{Pits} we show that this
assumption cannot be dropped.

Although we give up topological assumptions, primarily the
compacness, the results of this type turn out to be useful for
smooth dynamicas (see, e.g., \cite{BG}).

We use standard notation, terminology and results of entropy theory
(see, e.g., \cite{B} -- \cite{EM}, \cite{R}). Let $T$ be an
automorphism of a measurable space $(X,\mathcal F)$ and $\mu$ a
$T$-invariant probability measure. For a countable partition $\eta$
of $X$ we write $B\in\eta$ and $B\subset\eta$ if the set $B$ is an
atom of $\eta$ or a union of such atoms, respectively. We call
$\eta$ a generating partition (or a generator) for $(X,\mathcal
F,T,\mu)$ (usually just for $(T,\mu)$) if the $\sigma$-algebra
generated by the partitions $T^i\eta$, $i\in\mathbb Z$, coincides
$\mu$-mod 0 with $\mathcal F$. In what follows we usually omit
$\mathcal F$ from the notation.

\begin{theorem}
\label{main} Assume that for a countable measurable partition $\xi$
of $X$ and a $T$-invariant probability measure $\mu$, the entropy
$H_{\mu}(\xi)$ is finite and that there exist a sequence of
$T$-invariant probability measures $\mu_n$ and sequences of numbers
$r_n\in\mathbb N$, $\varepsilon_n>0$ such that
\begin{equation}
\label{limits} \lim_{n\to\infty}r_n=\infty,\ \
\lim_{n\to\infty}\varepsilon_n=0,
\end{equation}
\begin{equation}
\label{uslovie} |\mu(A)-\mu_n(A)|\le\varepsilon_n\mu_n(A)\text{\ for
all\ } A\in\vee_{i=0}^{r_n}T^{-i}\xi,\ n\ge1,
\end{equation}
\begin{equation}
\label{generator} \xi\text{\ is a generator for }\ (T,\mu) \text{
and } (T,\mu_n), n\ge1.
\end{equation}
Then
\begin{equation}
\label{conclud} h_{\mu}(T)\ge\limsup_{n\to\infty}h_{\mu_n}(T).
\end{equation}

\end{theorem}

This theorem will be proved in Section 2. In Section 3 we show that
the assumption $H_{\mu}(\xi)<\infty$ cannot be in general omitted
and in Section 4 establish another estimate for $h_{\mu}(T)$ without
the requirement that $H_{\mu}(\xi)<\infty$.
\bigskip

\section{Proof of theorem \ref{main}}
\label{Proof}
\medskip

We begin the proof of Theorem \ref{main} with two simple lemmas.
\begin{lemma}
\label{entrop_razn} Let $\mathbf p:=(p_i)_{i\in\mathbb N}$, $\mathbf
q:=(q_i)_{i\in\mathbb N}$, where $p_i,q_i\ge 0$ for all $i$ and
$\sum_ip_i=\sum_iq_i=1$. Let also $H(\mathbf p):=-\sum_{i\in\mathbb
N}p_i\ln p_i$ (with $0\ln 0=0$) and $H(\mathbf q)$ be defined
similarly. Assume that for some $c\in(0,1/3)$,
\begin{equation}
\label{sravn_ver} |p_i-q_i|\le cq_i,\ \ i=1,2,\dots.
\end{equation}
Then
\begin{equation*}
\label{entrop_razn1} H(\mathbf p)\le(1+c)H(\mathbf q)+c\ln 3.
\end{equation*}
\end{lemma}
\begin{proof}
Denote $\varphi(t):=-t\ln t$, $t\ge 0$. It is clear that (a)
$\varphi(t)$ increases when $0\le t\le e^{-1}$, (b) $\varphi(t)\le
0$ when $t\ge 1$, (c) $-1\le\varphi'(t)\le\ln 3$ when $(3e)^{-1}\le
t\le 1$. Hence (see also (\ref{sravn_ver}))
\begin{align*}
H(\mathbf p)=&\sum_{i\in\mathbb N}\varphi(p_i)=\sum_{i:q_i\le
1/2e}\varphi(p_i)+\sum_{i:q_i>1/2e}\varphi(p_i)\le\sum_{i:q_i\le
1/2e}\varphi((1+c)q_i)\notag \\
+&\sum_{i:q_i>1/2e}[\varphi(q_i)+|p_i-q_i|\ln 3]=\sum_{i:q_i\le
1/2e}[q_i\varphi(1+c)+(1+c)\varphi(q_i)] \notag \\
+&\sum_{i:q_i>1/2e}[\varphi(q_i)+|p_i-q_i|\ln
3]\le(1+c)\sum_{i\in\mathbb N}\varphi(q_i)+c\ln 3 \notag \\
=&(1+c)H(\mathbf q)+ c\ln3.
\end{align*}
\end{proof}
\begin{lemma}
\label{AinEta} If $\eta$ is a countable measurable partition of the
space $X$ and if, for probability measures $\mu$ and $\nu$ on
$(X,\mathcal F)$, every $A\in\eta$ and some $\varepsilon>0$, we have
$|\mu(A)-\nu(A)|\le\varepsilon\nu(A)$, then the same is true for
every $A\subset\eta$.
\end{lemma}

The proof is evident and is omitted.
\medskip

We continue the proof of Theorem \ref{main}. For all $k,l\in\mathbb
Z$, $k\le l$, we denote $\xi_k^l(T):=\vee_{i=k}^lT^i\xi$.

It is known \cite{R}, \cite{EM} that if $H_{\mu_0}(\xi)<\infty$ and
(\ref{generator}) holds for $n=0$, then
\begin{equation*}
\label{lim_entrop}
h_{\mu_0}(T)=\lim_{n\to\infty}\frac{1}{n}H_{\mu_0}(\xi_{-n}^{-1}(T)),
\end{equation*}
and the sequence on the right hand side is non-increasing.

It easy to verify that if $\varepsilon\le 1/2$, then
$|a-b|\le\varepsilon b$, $a,b\ge 0$ imply that $|a-b|\le
2\varepsilon a$. Therefore by (\ref{uslovie}) for all $n$ such that
$\varepsilon_n\le1/2$ and all $A\in\xi(-r_n,0)$, we have
\begin{equation}
\label{uslovie1} |\mu_0(A)-\mu_n(A)|\le2\varepsilon_n\mu_0(A).
\end{equation}

For such $n$ we compare $H_{\mu_0}(\xi_{-r_n}^{-1}(T))$ and
$H_{\mu_n}(\xi_{-r_n}^{-1}(T))$.

An arbitrary numbering of the atoms $A\in\xi_{-r_n}^{-1}(T)$ yields
a sequence $A_1,A_2,\dots$. Let $p_i:=\mu_n(A_i)$,
$q_i:=\mu_0(A_i)$. By applying Lemma \ref{entrop_razn} with
$c=2\varepsilon_n$ (see \eqref{uslovie1}) we obtain
\begin{equation*}
\label{entrop_razn2}
H_{\mu_n}(\xi_{-r_n}^{-1}(T))\le(1+2\varepsilon_n)
H_{\mu_0}(\xi_{-r_n}^{-1}(T))+2\varepsilon_n\ln 3.
\end{equation*}

For all sufficiently large $n$, this implies that
\begin{align*}
\label{osn_ner}
h_{\mu_n}(T,\xi)\le&\frac{1}{r_n}H_{\mu_n}(\xi_{-r_n}^{-1}(T)) \notag\\
\le&\frac{1}{r_n}(1+2\varepsilon_n)H_{\mu_0}(\xi_{-r_n}^{-1}(T))+
\frac{2}{r_n}\varepsilon_n\ln 3,
\end{align*}
or
\begin{equation*}
\label{osn_ner1} h_{\mu_n}(T,\xi)\le
\frac{1}{r_n}(1+2\varepsilon_n)H_{\mu_0}(\xi_{-r_n}^{-1}(T))+
\frac{2}{r_n}\varepsilon_n\ln 3.
\end{equation*}
Therefore (see (\ref{limits}))
$$
h\le\limsup_{n\to\infty}h_{\mu_n}(T,\xi)\le h_{\mu_0}(T).
$$
The proof is completed.

\section{Pitskel's example}
\label{Pits}
\medskip

Here we use an example constructed by B.\,Pitskel \cite{P} for a
different purpose.

Let $Y=\mathbb N^{\mathbb Z}$ be the sequence space equipped with
the cylinder $\sigma$-algebra $\mathcal C$ and $\tau$ the right one
step shift on $Y$ defined by $\tau y=y'$, where
$$
y=(y_i,\ i\in\mathbb Z),\ y'=(y'_i,\ i\in\mathbb Z),\ y'_i=y_{i-1},\
i\in\mathbb Z.
$$
We denote the partition of $Y$ into the one-dimensional cylinders
$C_k=\{y\in Y:y_0=k\}$, $k\in\mathbb N$, by $\eta$ and introduce the
$\tau$-invariant product measure $\nu$ on $Y$ by $\nu(C_k)=1/2^k$.
Then $(Y,\nu,\tau)$ is a Bernoulli shift and
$h_\nu(\tau)=H_\nu(\eta)<\infty$.

We now define a partition $\zeta>\eta$ as follows: for every
$k\in\mathbb N$, each atom of $\zeta$ lying in $C_k$ is an atom of
$\eta_{-2^k}^0(\tau)$.

One can easily check that
$H_\nu(\zeta|(\zeta_{-\infty}^{-1}(\tau))=H_\nu(\eta)=h_\nu(\tau)$.
But the following is true.
\begin{prop}
\label{cond_entrop} For every $n\in\mathbb N$,
\begin{equation}
\label{uslentrop} H_\nu(\zeta|(\zeta_{-n}^{-1}(\tau))=\infty.
\end{equation}
\end{prop}
\begin{proof}
By definition, for an arbitrary atom $A$ of $\zeta_{-n}^{-1}(\tau)$,
we have $A=\tau^{-1}A_1\cap\tau^{-2}A_2\cap\dots\cap\tau^{-n}A_n$,
where $A_i$ is an atom of $\zeta$ and $A_i\subset C_{k_i}$ for some
$k_i\in\mathbb N$. Moreover, for each $i$,
\begin{equation}
\label{Atom} A_i=C_{m_i(0)}\cap\tau^{-1}C_{m_i(1)}\cap\dots\cap
\tau^{-2^{k_i}}C_{m_i(2^{k_i})},
\end{equation}
where $m_i(0)=k_i$ and $m_i(j)\in\mathbb N$, $1\le j\le 2^{k_i}$.
Hence
\begin{equation}
\label{shiftA} \tau^{-i}A_i=\tau^{-i}C_{m_i(0)}\cap
\tau^{-i-1}C_{m_i(1)}\cap\dots\cap\tau^{-i-2^{k_i}}C_{m_i(2^{k_i})},\
\ 1\le i\le n.
\end{equation}

In order to evaluate $H(\zeta|(\zeta_{-n}^{-1}(\tau)_{-n}^0)$, we
will describe the partition induced by $\zeta$ on $A$. Let
$$
l(A):=\max_{1\le i\le n}(i+2^{k_i}).
$$
Then $A$ is a cylinder defined by fixing the coordinates $y_r$ with
$-l(A)\le r\le -1$. For every atom $A'$ of $\zeta$, we have
$A'\subset C_{k'}$, where $k'\in\mathbb N$. Then $A'$ is a cylinder
defined by fixing coordinates $y_r$ with $-2^{k'}\le r\le 0$.
Therefore, if $2^{k'}\le l(A)$, then
\begin{equation}
\label{intersec} A\cap A'=\emptyset\text{ or }A\cap C_{k'},
\end{equation}
and if $2^{k'}>l(A)$, then
\begin{equation}
\label{intersec1}A\cap A'=\emptyset\text{ or } A\cap C_{m'(0)}\cap
\tau^{-l(A)-1}C_{m'(1)}\cap\dots\cap\tau^{-2^{k'}}C_{m'(2^{k'})},
\end{equation}
where $m'(0)=k'$,  $m'(r)\in\mathbb N$ for $1\le r\le 2^{k'}$ (see
(\ref{Atom}, (\ref{shiftA})).

By definition
\begin{equation}
\label{cond_ent1} H_\nu\left(\zeta|\zeta_{-n}^{-1}(\tau)\right)=
\sum_{A\in\zeta_{-n}^{-1}(\tau)}\nu(A)H_\nu(\zeta|A).
\end{equation}
Using (\ref{intersec1}), the independence of the partitions
$\tau^i\eta$ for different indices $i$, and standard properties of
entropy, we obtain
\begin{align}
\label{cond_ent2} H_\nu(\zeta|A)\ge&\sum_{k\ge
1}\mu(C_k|A)H_\nu(\zeta|C_k\cap
A)\ge\sum_{k:2^k>l(A)}\nu(C_k|A)H_\nu(\zeta|C_k\cap A)\notag \\
=\sum_{k:2^k>l(A)}&\frac{1}{2^k}H_\nu\left(\eta_{-2^k}^{-l(A)-1}(\tau)
\right)=\sum_{k:2^k>l(A)}\frac{1}{2^k}(2^k-l(A))H_\nu(\eta)=\infty.
\end{align}
Now (\ref{cond_ent1}) and (\ref{cond_ent2}) yield (\ref{uslentrop}).
\end{proof}

We now label the atoms of $\zeta$ by positive integers, introduce
the set $X:=\mathbb N^\mathbb Z$ and, for every point $y\in Y$,
write $\varphi(y)_i=k$ if $\tau^{-i}y$ belongs to the atom of
$\zeta$ labeled by $k$, $k\in\mathbb N$, $i\in\mathbb Z$. The
mapping $\varphi:Y\to X$ so defined is clearly an embedding. Let $T$
be the shift on $X$ (similar to $\tau$) and $\mu:=\varphi\circ\nu$.
Then the dynamical systems $(Y,\nu,\tau)$ and $(X,mu,T)$ are
isomorphic. Hence
\begin{equation}
\label{entrop_equal} h_\mu(T)= h_{\nu}(\tau)=H_\nu(\eta)<\infty.
\end{equation}

Let $\xi$ be the partition of $X$ into the sets $\{x\in X:x_0=k\}$,
$k\in\mathbb N$ (the 1-dimensional cylinders in $X$ with support
$\{0\}$). For $n=1,2,\dots$ we define a measure $\mu_n$ on $X$ such
that

(a) it is $T$-invariant,

(b) it satisfies $\mu_n(A)=\mu(A)$ for all $A\in \xi_{-n}^0(T)$,

(c) it is a Markov measure of order $n$.

Property (c) means that, for all $A\in\xi$, all $m\in\mathbb N$, and
all $A'\in\xi_{-n-m}^{-1}(T)$ with $\mu(A')>0$, we have
$\mu_n(A|A')=\mu_n(A|A'')$, where $A''$ is the atom of
$\xi_{-n}^{-1}(T)$ containing $A$. It is easy to check that there
exists a unique measure with properties (a) -- (c). We call it the
$n$-{\it Markov hull} of $\tilde\mu$.

From (\ref{uslentrop}) it follows that
\begin{equation}
\label{entrop_equal1} H_{\mu_n}(\xi|\xi_{-\infty}^{-1}(T))=
H_{\mu_n}(\xi|\xi_{-n}^{-1}(T))=H_\nu(\zeta|\zeta_{-n}^{-1}(\tau))=
\infty.
\end{equation}

We claim that
\begin{equation}
\label{entrop_equal2}H_{\mu_n}(\xi|\xi_{-\infty}^{-1}({\tau}))=
h_{\mu_n}(T).
\end{equation}
Indeed, if the entropy $H_{\mu_n}(\xi)$ of the generator $\xi$ for
$(T,\mu_n)$ were finite, (\ref{entrop_equal2}) would be a standard
property of the entropy. Otherwise it does not generally hold.
However, Pitskel \cite{P} observed that (\ref{entrop_equal2}) does
hold for every Markov measure of order 1, be the entropy of the
generator finite or not. From this it follows that
(\ref{entrop_equal2}) holds for Markov measures of all orders, so
that $h_{\mu_n}(T)=\infty$ for all $n$.

On the other hand, it is seen that if we put $r_n=n$, then
conditions (\ref{limits}) -- (\ref{generator}) will hold for any
$\varepsilon_n\to 0$. However, by (\ref{entrop_equal}) --
(\ref{entrop_equal2}) the conclusion (\ref{conclud}) fails.

\section{Another estimate for $h_{\mu_0}(T)$}
\label{Another} We number in an arbitrary way the atoms of $\xi$,
pick an increasing sequence of positive integers $q(m)$,
$m=1,2,\dots$, and replace all the atoms with labels $\ge q(m)$ by
their union. The partition thus obtained will be denoted by $\xi_m$
and its atoms by $A_1,\dots,A_{q(m)}$.
\begin{theorem}
\label{two} If all the conditions of Theorem \ref{main} with one
possible exception, namely, $h_{\mu}(\xi)<\infty$, are fulfilled,
then
\begin{equation}
\label{limsup} h_{\mu}(T)\ge\limsup_{m\to\infty}
\limsup_{n\to\infty}h_{\mu_n}(T,\xi_m).
\end{equation}
\end{theorem}
\begin{proof}
We fix $m$ and introduce the sequence space
$Y_m:=\{1,\dots,q(m)\}^\mathbb Z$ with a standard metric $\rho_m$,
the Borel $\sigma$-algebra $\mathcal B_m$, and the one step right
shift $\sigma_m$ on $Y_m$. Then we define a mapping $\psi_m:X\to
Y_m$ by $\psi_mx=y$, where $y_n=i$ provided that $T^{-n}x\in A_i$.
It is easy to check that $\psi_mTx=\sigma_m\psi_mx$ for all $x\in
X$.

For every $m\in\mathbb N$ and $n\in\mathbb Z^+$ consider the
probability measures $\mu^m:=\psi_m\circ\mu$ and
$\mu_n^m:=\psi_m\circ\mu_n$ on $(Y_m,\mathcal B_m)$. From
(\ref{uslovie}) it follows that $\mu_n^m$ converges weakly to
$\mu^m$ as $n\to\infty$. Since the space $(Y_m,\rho_m)$ is compact
and $\sigma_m$ is expansive, we have
$$
h_{\mu^m}(\sigma_m)\ge\limsup_{n\to\infty}h_{\mu_n^m}(\sigma^m),\ \
m\in\mathbb N.
$$
But it is clear that
$$
h_{\mu^m}(\sigma_m)=h_{\mu}(T,\xi_m),\ \
h_{\mu_n^m}(\sigma_m)=h_{\mu_n}(T,\xi_m),\ \ m\in\mathbb N,\ \
n\in\mathbb Z^+.
$$
Hence
$$
h_\mu(T,\xi_m)\ge\limsup_{n\to\infty}h_{\mu_n}(T,\xi_m),\ \
m\in\mathbb N.
$$
From well-known properties of entropy it follows that
$h_\mu(T,\xi_m)$ tends to $h_\mu(T)$ as $m\to\infty$. This implies
(\ref{limsup}).
\end{proof}
In conclusion we give a simple example where the assumptions of
Theorem \ref{two} are satisfied and this theorem yields the correct
value of $h_\mu(T)$.
\begin{example}
\label{th2} Let $Y$, $\mathcal C$, $\tau$, and $\eta$ be as in
Section \ref{Pits}, and let $\nu$ be a $\tau$-invariant probability
measure on $(Y,\mathcal C)$. Consider a function $f:Y\to\mathbb Z_+$
with $\nu(f)<\infty$ such that $f$ is constant on each atom of
$\eta$. The integral (suspension) automorphism $\hat\tau$
constructed by $\tau$ and $f$ acts on the space $\hat
Y:=\{(y,u):x\in Y,u\in\mathbb Z_+,u\le f(x)\}$ by the formulas
$\hat\tau(y,u)=(y,u+1)$ if $u<f(y)$ and $\hat\tau(y,u)=(\tau y,0)$
if $u=f(y)$. Clearly $\hat\tau$ preserves the probability measure
$\hat\nu:=\frac{1}{\nu(f)+1}(\nu\times\lambda)_{\hat Y}$, where
$\lambda$ is the counting measure on $\mathbb Z_+$.

Let $\eta=\{C_1,C_2,\dots\}$ and $f(y)=f_i$ when $y\in C_i$. By
construction $\sum_{n=1}^\infty f_i\nu(C_i)<\infty$. We may also
assume that $\sum_{n=1}^\infty f_i\nu(C_i)|\log\nu(C_i)|=\infty$.
Then $H_{\hat\nu}(\hat\eta)=\infty$, where $\hat\eta$ is the
partition of $\hat Y$ whose atoms are $C_{i,k}=\{(y,u):y\in
C_i,u=k\}$, $i\in\mathbb N$, $k=0,1,\dots,f_i$. Clearly $\hat\eta$
is a countable generator for $(\hat\nu,\hat\tau)$.

We denote the $n$-Markov hull of $\mu$ (see Section \ref{Pits} for a
definition) by $\mu_n$ and define a measure $\hat\mu_n$ on $\hat X$
by
$$
\hat\mu_n|_{T^iC_k}:=\frac{1}{\mu(f)+1}\hat T^i\circ(\mu|_{C_k}),\ \
k\in\mathbb N,\ \ 0\le i\le f_k.
$$
Let $\hat\xi_m$ be the partition of $\hat X$ whose atoms are
$C_{i,k}$, $1\le i\le m$, $0\le k\le f_i$, and the union of
$C_{i,k}$ over all $i>m$ and $0\le k\le f_i$.

Taking $\hat X$, $\hat\mu$, $\hat T$, $\hat\xi$, and $\hat
r_m:=\sum_{i=1}^m (f_i+1)$ for $X$, $\mu$, $T$, $\xi$, and $r_m$,
respectively, in Theorem \ref{two} one can easily check that the
condition of this theorem  are satisfied for any $\varepsilon_n\to
0$.
\end{example}

\end{document}